\documentclass[reqno,11pt, letterpaper]{amsart}
\PassOptionsToPackage{fleqn}{amsmath}
\RequirePackage{amsthm}
\RequirePackage{amsmath}
\RequirePackage{latexsym}
\RequirePackage{amssymb}
\RequirePackage{amscd}
\RequirePackage{epsfig}
\RequirePackage{graphics}
\RequirePackage{ifthen}
\RequirePackage{varioref}
\usepackage{changes}
\usepackage{soul}
\usepackage{xcolor}
\usepackage{color}
\usepackage[misc]{ifsym}
\usepackage{bm}
\usepackage{lipsum}
\usepackage{subfigure}
\usepackage{enumitem}
\usepackage{changepage}
\usepackage{mathrsfs}

\usepackage{epstopdf}

\definecolor{darkgreen}{rgb}{0.0625,0.64,0.0625}
\ifpdf
  \usepackage[
    pdftex,
    colorlinks,%
    linkcolor=blue,citecolor=blue,urlcolor=blue,
    hyperindex,%
    plainpages=false,%
    bookmarksopen,%
    bookmarksnumbered%
  ]{hyperref}
\else
  \usepackage{hyperref}
\fi
\usepackage{bookmark}
\allowdisplaybreaks[1]

\theoremstyle{plain}
\newtheorem{thm}{Theorem}[section]
\newtheorem{lem}[thm]{Lemma}
\newtheorem{prop}[thm]{Proposition}
\newtheorem{cor}[thm]{Corollary}
\theoremstyle{definition}
\newtheorem{rem}[thm]{Remark}

\numberwithin{equation}{section}
\newcommand\blfootnote[1]{%
 \begingroup
 \renewcommand\thefootnote{}\footnote{#1}%
 \addtocounter{footnote}{-1}%
 \endgroup
 }

\begin{document}

\title[Centro-Equiaffine Area-Preserving Flow]
{A fourth-order area-preserving curve flow in centro-equiaffine geometry}

	\author[X. Jiang]{Xinjie Jiang ${}^*$}\blfootnote{${}^*$~Corresponding author: serge0912@icloud.com}
\address{Xinjie Jiang}
\email[Corresponding author]{serge0912@icloud.com}

\author[S. Pan]{Shengliang Pan}
\address{Shengliang Pan}
\email{slpan@tongji.edu.cn}

\address{
School of Mathematical Sciences, Key Laboratory of Intelligent Computing and Applications (Ministry of Education), Tongji University, No.1239, Siping Road, Shanghai, 200092, P. R. China}

 \author[Y. Zhang]{Yanlong Zhang}
\address{ Yanlong Zhang\newline\indent
	Institute of Mathematics, Henan Academic of Sciences, Zhengzhou 450046, P. R. China}
\email{ylzhang@hnas.ac.cn}

\begin{abstract}
	In this paper, inspired by Guan and Li~\cite{GL14}, we introduce a fourth-order centro-equiaffine invariant curve flow via the affine Minkowski formula. Without any smallness assumptions on the initial curve, we establish the long-time existence of the flow and prove that, as $t \to +\infty$, the evolving curve preserves its enclosed area and converges smoothly to a round circle up to the action of $\mathrm{SL}(2)$.
\end{abstract}

\subjclass[2020]{53A15, 53A55, 53E40, 35K52}

\keywords{centro-equiaffine geometry;\; invariant curve flow;\; higher-order curvature flow;\; asymptotic behavior}


\maketitle
\section{Introduction}
In this paper, we study a centro-equiaffine invariant constrained curve flow, which can be viewed as the natural planar affine analogue of the mean curvature type flow introduced by Guan and Li \cite{GL14}. 
\textit{Constrained curvature flows} constitute an important class of geometric evolutions, characterized by the preservation of certain global geometric quantities along the flow.
A classical example is the area-preserving curve shortening flow introduced by Gage \cite{Gag85}, where a nonlocal term is incorporated into the velocity to fix the enclosed area. 
Gage proved that this flow preserves the area and deforms any convex initial curve smoothly to a round circle. 
The idea of introducing appropriate nonlocal terms to control global quantities has since been widely generalized, see \cite{TW15,GPT20.1,GPT20.2,GPT21} and references therein.
In higher dimensions, the corresponding analogue is the volume-preserving mean curvature flow studied by Huisken \cite{Hui87}, who showed that strictly convex hypersurfaces exist for all time and converge smoothly to round spheres. 
Similar results in hyperbolic space were obtained under a suitable convexity assumption \cite{CM07}.

However, the presence of global terms typically introduces substantial analytical difficulties, and apart from certain perturbative results \cite{AF03}, general results in the spherical setting remain limited. 
Guan and Li \cite{GL14} observed that, by exploiting the Minkowski identities in space forms, one can define a locally constrained mean curvature type flow that involves no global term. This flow preserves the enclosed volume while monotonically decreasing the surface area. Compared with the classical volume-preserving mean curvature flow, this formulation requires significantly weaker assumptions on the initial hypersurface. They proved that, when the ambient manifold is a space form, the flow exists for all time and evolves star-shaped initial hypersurfaces smoothly to round spheres. This result was later extended to warped product spaces in joint work with Wang \cite{GLW19}.

Affine differential geometry is based on the Lie group $\mathrm{Aff}(n)=\mathrm{GL}(n)\ltimes \mathbb{R}^n$, consisting of affine transformations of the form  $\mathrm{x}\mapsto \mathcal{A}\mathrm{x}+b$, where $\mathcal{A}\in \mathrm{GL}(n)$ and $b\in \mathbb{R}^n$ (see \cite{NS94,Sim00} for details). 
Restricting to the subgroup $\mathrm{SA}(n)=\mathrm{SL}(n)\ltimes \mathbb{R}^n$ of volume-preserving affine transformations leads to equi-affine geometry. 
Centro-affine differential geometry concerns affine transformations fixing the origin and is closely related to the geometry induced by the general linear group $\mathrm{x}\mapsto \mathcal{A}\mathrm{x}$, $\mathcal{A}\in \mathrm{GL}(n)$. 
Centro-equiaffine differential geometry corresponds to the subgroup $\mathrm{SL}(n)$ of volume-preserving linear transformations.

Our analysis is carried out mainly within the framework of planar equi-affine differential geometry. Let $\gamma:S^1 \to \mathbb{R}^2$ be a strictly convex embedded curve parametrized by $x$. The equi-affine metric is defined by $g=[\gamma_x,\gamma_{xx}]^{\frac13}$, where $[\cdot,\cdot]$ denotes the standard determinant in $\mathbb{R}^2$. The corresponding equi-affine arc-length parameter is $s(x):=\int_0^x g(\eta)\, d\eta$. The affine support function is defined by
\begin{align*}
	\sigma := [\gamma, \gamma_s],
\end{align*}
and the equi-affine curvature is given by
\begin{align*}
	\mathcal{K} := [\gamma_{ss}, \gamma_{sss}].
\end{align*}
These quantities satisfy the classical relation
\begin{align}\label{M-relation}
	\sigma_{ss} + \sigma \mathcal{K} = 1.
\end{align}

Over the past three decades, substantial progress has been made on invariant geometric flows associated with various affine subgroups. 
In equi-affine geometry, Sapiro and Tannenbaum \cite{ST94} introduced the affine curve shortening flow, which serves as the equi-affine analogue of the classical curve shortening flow, and showed that a closed, convex embedded curve evolving under this flow shrinks to an elliptical point. 
Since this flow does not decrease the equi-affine length pointwise, Andrews \cite{And99} proposed an $L^2$-gradient flow of the equi-affine length, proving that the flow expands to infinity and asymptotically approaches a homothetically expanding ellipse. 
A higher-dimensional analogue of the affine curve shortening flow is the affine normal flow. Andrews \cite{And96} proved that, under this flow, strictly convex hypersurfaces contract to ellipsoidal points. In the context of centro-equiaffine geometry, Stancu \cite{Sta12} introduced an invariant $p$-flow to construct new global invariants of smooth convex bodies. 
Using the short-time existence of the flow, she derived a series of isoperimetric-type inequalities. 
Ivaki \cite{Iva14} further investigated the asymptotic behavior, showing that for origin-symmetric initial convex curves the flow shrinks to a point in finite time, while the normalized flow converges in the Hausdorff metric to the unit circle up to the action of $\mathrm{SL}(2)$. 
The case $p=\infty$, corresponding to the centro-equiaffine geometric heat flow, was studied in \cite{Iva16,WWQ17}, while the higher-dimensional asymptotic behavior of the $p$-flow was investigated in \cite{IS13,Iva15}. In centro-affine geometry, Qu and Yang \cite{QY22} proposed a second-order invariant curve flow, whose asymptotic behavior was analyzed by Yang, Yu, and the first author \cite{JYY23}. 
This work was subsequently extended to higher dimensions by the first author, Qu, and Yang \cite{JQY26}. 
Niu and Yang \cite{NY25} studied a third-order invariant curve flow preserving both the centro-affine length and the centro-affine elastic energy. 
In the general-affine setting, Yang \cite{Yan23} investigated an invariant curve flow from the perspective of geometric heat flow.

From \eqref{M-relation}, it follows that for a closed curve one has the following well-known affine Minkowski integral formula (cf.\ \cite{NS94})
\begin{align*}
	\oint\left(1-\sigma\mathcal{K}\right)ds=0.
\end{align*}
Motivated by this formula, we consider the following Guan--Li type centro-equiaffine invariant curve flow 
\begin{equation}\label{main-flow1}
	\left\{
	\begin{aligned}
		&\frac{\partial\gamma}{\partial t}=(1-\sigma\mathcal{K})\mathcal{N},\\ &\gamma(\cdot,0)=\gamma_0,
	\end{aligned}
	\right.
\end{equation}
where $\mathcal{N}$ is the equi-affine normal vector of the curve $\gamma$, and $\gamma_0$ is a smooth, closed, origin-symmetric, strictly convex curve in $\mathbb{R}^2$. 
Using \eqref{M-relation}, the flow \eqref{main-flow1} can be rewritten as
\begin{equation}\label{main-flow2}
	\left\{
	\begin{aligned}
	&\frac{\partial\gamma}{\partial t}=\sigma_{ss}\mathcal{N},\\ &\gamma(\cdot,0)=\gamma_0.
	\end{aligned}
	\right.
\end{equation}
\begin{rem}
	Note that the affine support function $\sigma$ is not translation invariant. 
	In fact, it is related to the centro-equiaffine curvature $\mathscr{K}$ by (see \cite{Iva14})
	\[
	\mathscr{K} = \sigma^{-3}.
	\]
	Hence, the flow \eqref{main-flow1} is invariant under $\mathrm{SL}(2)$ transformations.
\end{rem}
The flow considered here is a fourth-order curvature flow, for which the maximum principle is no longer available. 
As a consequence, such flows often present significant analytical obstacles and can be highly delicate to analyze.
A prototypical example is the surface diffusion flow, introduced by Mullins to model phase–interface dynamics \cite{Mul57}, which can be interpreted as the $H^{-1}$-gradient flow of the surface area. 
In the planar setting, in contrast to second-order flows such as the classical curve shortening flow, the surface diffusion flow does not preserve convexity \cite{GI99} or embeddedness \cite{GI98}, properties that are fundamental in the analysis of second-order curve flows.
Due to these difficulties, additional assumptions are often required to obtain asymptotic results. 
Wheeler \cite{Whe13} proved that if the initial closed curve is sufficiently close to a round circle, in the sense of the normalized $L^2$ oscillation of curvature, then the flow exists for all time and converges exponentially fast to a circle. 
More recently, Miura and Wheeler \cite{MW25} established a similar result for the free elastic flow: if the initial curve is sufficiently close to a circle in terms of the normalized $L^2$ norm of the derivative of curvature, then a suitable rescaling of the flow converges smoothly to a unique round circle. 
Furthermore, Andrews et al.\ \cite{AMWW20} showed that for the ideal flow, a uniform bound on the length implies smooth convergence to a multiply covered circle; in particular, such a bound follows from a smallness condition on the normalized $L^2$ norm of the curvature derivative.

In the present work, we study the flow without any smallness assumption on the initial data. 
Our main result is the following.
\begin{thm}\label{A}
	Let $\gamma_0: S^1 \to \mathbb{R}^2$ be a smooth strictly convex curve that is symmetric with respect to the origin. 
	Then the flow \eqref{main-flow1} exists for all time, and $\gamma(\cdot,t)$ converges smoothly, as $t\to+\infty$, to a round circle with the same area as $\gamma_0$, modulo $\mathrm{SL}(2)$.
\end{thm}
This paper is organized as follows. Section \ref{sec:preliminaries} recalls the relevant definitions, notions, and geometric inequalities in affine plane geometry. The basic properties of the flow \eqref{main-flow1}, including the short-time existence, the evolution equations of geometric quantities, and several monotone quantities, are discussed in Section \ref{sec:basic}. In Section \ref{sec:sigma-bound}, we establish uniform bounds for the affine support function along the flow. Section \ref{sec:energy-estimates} is devoted to energy estimates. Finally, by combining these energy estimates, we establish the long-time existence of the flow and analyze its asymptotic behavior in Section \ref{Proof of Theorem A}.

\section{Preliminaries}\label{sec:preliminaries}
Let \(\gamma: S^1 \to \mathbb{R}^2\) be a smooth, strictly convex, embedded curve parametrized by \(x\).
The equi-affine metric is defined by
\[
g := [\gamma_x,\gamma_{xx}]^{\frac{1}{3}},
\]
where \([\cdot,\cdot]\) denotes the standard determinant on  \(\mathbb{R}^2\). 
The associated equi-affine arc-length parameter is
\[
s(x) := \int_0^x g(\eta)\,d\eta.
\]
The equi-affine tangent vector \(\mathcal{T}\), the equi-affine normal vector \(\mathcal{N}\), and the equi-affine curvature \(\mathcal{K}\) are defined as
\[
\mathcal{T} := \gamma_s, 
\qquad 
\mathcal{N} := \gamma_{ss}, 
\qquad 
\mathcal{K} := [\gamma_{ss},\gamma_{sss}].
\]
With respect to the equi-affine arc-length parameter \(s\), one has
\[
[\gamma_s,\gamma_{ss}] = 1, 
\qquad 
[\gamma_s,\gamma_{sss}] = 0, 
\qquad 
[\gamma_{ssss},\gamma_s] = \mathcal{K}.
\]
The equi-affine Serret–Frenet formulas are given by
\begin{equation}\label{structure}
	\mathcal{T}_s = \mathcal{N}, 
	\qquad 
	\mathcal{N}_s = -\mathcal{K}\,\mathcal{T}.
\end{equation}
A fundamental scalar quantity in affine differential geometry is the affine support function
\begin{equation*}\label{aff-supp}
	\sigma := [\gamma,\gamma_s].
\end{equation*}
By a direct computation, one can verify that $\sigma$ and the equi-affine curvature $\mathcal{K}$ are related by
\[
\sigma_{ss} + \sigma \mathcal{K} = 1.
\]
The global geometric quantities of $\gamma$, namely the equi-affine length $\mathcal{L}$ and the enclosed area $A$, can be expressed as
\begin{align}\label{length-area}
	\mathcal{L} = \oint ds = \oint \sigma \mathcal{K}\, ds, 
	\qquad 
	A = \frac{1}{2} \oint \sigma\, ds.
\end{align}

\subsection{Affine Geometric Inequalities}

We recall several fundamental affine geometric inequalities for convex curves, which will be used later.

\begin{prop}[Equi-affine isoperimetric inequality]\label{Iso ine}
	For any smooth strictly convex curve $\gamma$,
	\[
	\mathcal{L} \le 2\pi^{\frac{2}{3}} A^{\frac{1}{3}},
	\]
	with equality if and only if $\gamma$ is an ellipse.
\end{prop}

\begin{prop}[Blaschke--Santal\'o inequality]\label{BS ine}
	For any smooth strictly convex curve $\gamma$, there exists a point $p_0 \in \mathbb{R}^2$ in the interior of $\gamma$ such that
	\[
	A \oint \frac{1}{\sigma^2}\, ds \le 2\pi^2,
	\]
	where $\sigma$ is the affine support function with respect to $p_0$. 
	Equality holds if and only if $\gamma$ is an ellipse.
\end{prop}

\begin{prop}[Aleksandrov--Fenchel inequality]\label{AF ine}
	For any smooth strictly convex curve $\gamma$,
	\[
	A \oint \mathcal{K}\, ds \le \frac{1}{2}\mathcal{L}^2,
	\]
	with equality if and only if $\gamma$ is an ellipse.
\end{prop}

\section{Basic properties of the flow}\label{sec:basic}
We begin with the short-time existence of the flow \eqref{main-flow1}. 
To this end, we reformulate the flow as a scalar parabolic equation via the Euclidean support function.

Let $\gamma: S^1 \to \mathbb{R}^2$ be a smooth, strictly convex, embedded closed curve parametrized by the inverse Gauss map. 
Its Euclidean support function is defined by
\[
h(z) := \langle \gamma(z), z \rangle, \quad z \in S^1.
\]
The curve can be reconstructed from $h$ via
\[
\gamma(z) = h(z) z + h_\theta(z) z_\theta,
\]
where $\theta$ denotes the angular parameter on $S^1$. 
Moreover, the radius of curvature is given by
\[
r[h] = h_{\theta\theta} + h.
\]
Let $\mathfrak{s}$ denote the Euclidean arc-length parameter, $\mathbf{t} = \gamma_\mathfrak{s}$ the unit tangent vector, and $\mathbf{n}$ the outward unit normal vector. In the Euclidean framework, the equi-affine arc length reads
\begin{align}\label{elements}
	ds = \kappa^{\frac{1}{3}}\, d\mathfrak{s} = r^{\frac{2}{3}}\, d\theta,
\end{align}
with the affine support function expressed as
\begin{equation}\label{support-relation}
	\sigma = h\, r^{\frac{1}{3}}.
\end{equation}
The affine tangent and normal vectors take the form
\begin{align}\label{vector-relation}
	\mathcal{T} &= \kappa^{-\frac{1}{3}} \mathbf{t}, \nonumber \\
	\mathcal{N} &= -\kappa^{\frac{1}{3}} \mathbf{n} - \frac{1}{3} \kappa^{-\frac{5}{3}}\kappa_\mathfrak{s} \mathbf{t},
\end{align}
and the affine curvature becomes
\begin{align}\label{curva-relation}
	\mathcal{K} = \frac{1}{3} \kappa^{-\frac{5}{3}}\kappa_{\mathfrak{s}\mathfrak{s}}
	- \frac{5}{9} \kappa^{-\frac{8}{3}}\kappa_\mathfrak{s}^2
	+ \kappa^{\frac{4}{3}}=r^{-1}\left(\left(r^{-\frac{1}{3}}\right)_{\theta\theta}+r^{-\frac{1}{3}}\right).
\end{align}
Combining \eqref{support-relation}--\eqref{curva-relation}, the flow \eqref{main-flow1} in the Euclidean setting can be written as
\begin{equation*}
	\left\{
	\begin{aligned}
		&\frac{\partial \gamma}{\partial t} =
		\Big(
		\frac{1}{3} h \kappa^{-2} \kappa_{\mathfrak{s}\mathfrak{s}}
		- \frac{5}{9} h \kappa^{-3} \kappa_\mathfrak{s}^2
		+ h \kappa - 1
		\Big)
		\Big(
		\kappa^{\frac{1}{3}} \mathbf{n} + \frac{1}{3} \kappa^{-\frac{5}{3}} \kappa_\mathfrak{s} \mathbf{t}
		\Big),\\
		&\gamma(\cdot,0) = \gamma_0.
	\end{aligned}
	\right.
\end{equation*}
The above initial value problem is equivalent to a scalar parabolic equation for the support function $h$:
\begin{equation}\label{scalar-eq}
	\left\{
	\begin{aligned}
		&\frac{\partial h}{\partial t} 
		= -\frac{h h_{\theta\theta\theta\theta} + h h_{\theta\theta}}{3 (h_{\theta\theta} + h)^{7/3}}
		+ \frac{4 h (h_{\theta\theta\theta} + h_\theta)^2}{9 (h_{\theta\theta} + h)^{10/3}}
		+ \frac{h}{(h_{\theta\theta} + h)^{4/3}}
		- \frac{1}{(h_{\theta\theta} + h)^{1/3}},\\
		&h(z,0)=h_0(z),
	\end{aligned}
	\right.
\end{equation}
where $z\in S^1$ and $h_0$ is the support function for $\gamma_0$.
Equation \eqref{scalar-eq} is a nonlinear fourth-order scalar equation, which is parabolic as long as $h>0$ and $h_{\theta\theta}+h>0$ everywhere. 
By semigroup theory (see Proposition 2.3 of \cite{Lun90}, see also \cite{Lun95}), if the initial support function $h_0$ is smooth, positive, and satisfies $(h_0)_{\theta\theta}+h_0>0$, then there exists a unique, smooth solution of \eqref{scalar-eq} for a short time.  
This solution can then be used to reconstruct a unique solution of the original flow \eqref{main-flow1} for a short time by the same argument as in \cite{And98}, Lemma I1.2. Thus we conclude that
\begin{prop}[Local existence and uniqueness]\label{short-time}
	Assume that $\gamma_0$ is a smooth, closed, strictly convex curve in $\mathbb{R}^2$. Then there exists a time $T>0$ such that the flow \eqref{main-flow1} admits a unique smooth solution $\gamma(\cdot,t)$ satisfying $\gamma(\cdot,0)=\gamma_0$.
\end{prop}
As a direct consequence of uniqueness, symmetry is preserved.
\begin{prop}[Symmetry preservation]\label{symmetric}
	If the initial curve $\gamma_0$ is origin-symmetric, then the solution $\gamma(\cdot,t)$ remains origin-symmetric for all $t \in [0,T)$.
\end{prop}
\begin{proof}
	Note that both $-\gamma(\cdot+\pi,t)$ and $\gamma(\cdot,t)$ satisfy \eqref{main-flow1} with initial data $-\gamma(\cdot+\pi,0)$ and $\gamma(\cdot,0)$, respectively. At time $t=0$ we have $-\gamma(\cdot+\pi,0)=\gamma(\cdot,0)$. Therefore, by Proposition \ref{short-time} we conclude that $-\gamma(\cdot+\pi,t)=\gamma(\cdot,t)$ as long as the flow exists.
\end{proof}
\begin{prop}
	Ellipses are the only stationary solutions of the flow \eqref{main-flow1} among closed curves.
\end{prop}
\begin{proof}
	By \eqref{M-relation}, this is equivalent to solving
	\[
	\sigma_{ss} = 0,
	\]
	where $\sigma$ is a function of $s$ with period $\mathcal{L}$.  
	The general solution is easily obtained as
	\[
	\sigma(s) = C_1 s + C_2
	\]
	with constants $C_1$ and $C_2$.
	Imposing the periodicity condition $\sigma(s+\mathcal{L}) = \sigma(s)$ immediately yields $C_1 = 0$, hence $\sigma$ is constant: $\sigma \equiv C_2$.  
	Clearly, $\sigma$ cannot vanish identically. If $\sigma$ were a negative constant, the corresponding equi-affine curvature $\mathcal{K}$ would be negative.  
	By the classification of equi-affine plane curves \cite{Su83}, such a curve would be a hyperbola, which contradicts the assumption that the curve is closed.  
	Therefore, $\sigma$ must be a positive constant, and in this case the curve is an ellipse.
\end{proof}

\subsection{Evolution Equations and Monotone Quantities}
We next derive the evolution equations of several fundamental geometric quantities under a general equi-affine curve flow. Let $\gamma(\cdot,t)$ be a one-parameter family of smooth strictly convex closed plane curves evolving by
\begin{equation}\label{general-flow}
	\left\{
	\begin{aligned}
		&\frac{\partial \gamma}{\partial t}
		= \alpha \mathcal{N} + \beta \mathcal{T},\\
		&\gamma(\cdot,0) = \gamma_0,
	\end{aligned}
	\right.
\end{equation}
where $\alpha$ and $\beta$ are smooth scalar functions.
\begin{prop}\label{evo}
	Assume that the curve $\gamma(\cdot,t)$ evolves under the flow \eqref{general-flow}. Then the following evolution equations hold:
	\begin{itemize}[left=12pt, labelsep=5pt, itemsep=0pt]
		\item[$\mathrm{(i)}$]$\displaystyle\frac{g_t}{g}=\frac{1}{3}\alpha_{ss}-\frac{2}{3}\alpha\mathcal{K}+\beta_s$,\\
		\item[$\mathrm{(ii)}$]$\displaystyle\mathcal{K}_t=\frac{1}{3}\left(\alpha_{ssss}+5\alpha_{ss}\mathcal{K}+5\alpha_s\mathcal{K}_s+4\alpha\mathcal{K}^2+3\beta\mathcal{K}_s+\alpha\mathcal{K}_{ss}\right)$,\\
		\item[$\mathrm{(iii)}$]$\displaystyle\sigma_t=-\frac{1}{3}\left(\alpha\mathcal{K}+\alpha_{ss}\right)\sigma+\left(\alpha_s+\beta\right)\sigma_s-\alpha$,\\
		\item[$\mathrm{(iv)}$]$\displaystyle\mathcal{L}_t=-\frac{2}{3}\oint\alpha\mathcal{K}\,ds$,\\
		\item[$\mathrm{(v)}$]$\displaystyle A_t=-\oint\alpha\,ds$,\\
		\item[$\mathrm{(vi)}$]$\displaystyle\left(\oint\frac{1}{\sigma^2}\,ds\right)_t=2\oint\alpha\sigma^{-3}\,ds$,\\
		\item[$\mathrm{(vii)}$]$\displaystyle\left(\oint\mathcal{K}\,ds\right)_t=\frac{2}{3}\oint\alpha\left(\mathcal{K}_{ss}+\mathcal{K}^2\right)\,ds$.
	\end{itemize}
\end{prop}
\begin{proof}
	The identities \textup{(i)} and \textup{(ii)} are established in Appendix~A of \cite{Yan23}. 
	The remaining assertions follow from straightforward computations.
	We first compute the evolution of the affine support function $\sigma=[\gamma,\gamma_s]$. 
	Using \eqref{structure} together with \textup{(i)} and the commutation relation 
	$\partial_t\partial_s=\partial_s\partial_t-\frac{g_t}{g}\partial_s$, we obtain
	\begin{align*}
		\sigma_t
		&=\frac{\partial}{\partial t}[\gamma,\gamma_s]
		=[\gamma_t,\gamma_s]+[\gamma,\gamma_{st}]\\
		&=[\alpha\mathcal{N}+\beta\mathcal{T},\gamma_s]
		+[\gamma,\gamma_{ts}-\tfrac{g_t}{g}\gamma_s]\\
		&=-\alpha+(\alpha_s+\beta)[\gamma,\mathcal{N}]
		+\left(-\alpha\mathcal{K}+\beta_s-\tfrac{g_t}{g}\right)[\gamma,\mathcal{T}]\\
		&=-\alpha+(\alpha_s+\beta)\sigma_s
		-\frac{1}{3}\left(\alpha\mathcal{K}+\alpha_{ss}\right)\sigma,
	\end{align*}
	which proves \textup{(iii)}.
	Integrating \textup{(i)} over the curve immediately yields \textup{(iv)}:
	\[
	\mathcal{L}_t=\oint\frac{g_t}{g}\,ds
	=-\frac{2}{3}\oint\alpha\mathcal{K}\,ds.
	\]
	Next, recalling \eqref{length-area} and using \textup{(i)} and \textup{(iii)}, we derive
	\begin{align*}
	A_t
	&=\frac{1}{2}\oint\sigma_t\,ds+\frac{1}{2}\oint\sigma\frac{g_t}{g}\,ds\\
	&=\frac{1}{2}\oint\left(-\alpha\sigma\mathcal{K}
	+\alpha_s\sigma_s+\beta\sigma_s-\alpha+\sigma\beta_s\right)ds\\
	&=\frac{1}{2}\oint\left(-\alpha(\sigma\mathcal{K}+\sigma_{ss})-\alpha\right)ds\\
	&=-\oint\alpha\,ds,
\end{align*}
	which proves \textup{(v)}.
	Combining \textup{(i)} and \textup{(iii)}, we further obtain
	\begin{align*}
		\left(\oint\frac{1}{\sigma^2}\,ds\right)_t
		&=\oint\left(-2\sigma^{-3}\sigma_t+\sigma^{-2}\frac{g_t}{g}\right)ds\\
		&=\oint\left(\sigma^{-2}\alpha_{ss}
		-2\alpha_s\sigma^{-3}\sigma_s
		-2\beta\sigma^{-3}\sigma_s
		+2\alpha\sigma^{-3}
		+\beta_s\sigma^{-2}\right)ds\\
		&=2\oint\alpha\sigma^{-3}\,ds,
	\end{align*}
	which gives \textup{(vi)}.
	Finally, using \textup{(i)} and \textup{(ii)}, we compute
	\begin{align*}
		\left(\oint\mathcal{K}\,ds\right)_t
		&=\oint\left(\mathcal{K}_t+\mathcal{K}\frac{g_t}{g}\right)ds\\
		&=\oint\left(\frac{1}{3}\alpha_{ssss}
		+2\alpha_{ss}\mathcal{K}
		+\frac{5}{3}\alpha_s\mathcal{K}_s
		+\frac{2}{3}\alpha\mathcal{K}^2
		+\beta\mathcal{K}_s
		+\frac{1}{3}\alpha\mathcal{K}_{ss}
		+\beta_s\mathcal{K}\right)ds\\
		&=\frac{2}{3}\oint\alpha\left(\mathcal{K}_{ss}+\mathcal{K}^2\right)\,ds.
	\end{align*}
	This proves \textup{(vii)} and completes the proof.
\end{proof}

We now specialize to the flow \eqref{main-flow1} and deduce the following monotonicity properties.
\begin{prop}\label{monotone}
	Under the flow \eqref{main-flow1}, the evolving curve $\gamma(\cdot,t)$ preserves the enclosed area $A$, while the quantity $\displaystyle\oint\frac{1}{\sigma^2}\,ds$ is non-decreasing.
\end{prop}
\begin{proof}
	For \eqref{main-flow1}, we have $\alpha=1-\sigma\mathcal{K}=\sigma_{ss}$. Substituting this into Proposition~\ref{evo}\,\textup{(v)}, we obtain
	\[
	A_t=-\oint\sigma_{ss}\,ds=0.
	\]
	Similarly, applying Proposition~\ref{evo}\,\textup{(vi)} and integrating by parts yields
	\begin{align}\label{BS-quantity}
		\left(\oint\frac{1}{\sigma^2}\,ds\right)_t
		=2\oint\sigma_{ss}\sigma^{-3}\,ds
		=6\oint\sigma^{-4}\sigma_s^2\,ds
		\ge 0,
	\end{align}
	which proves the claim.
\end{proof}

\section{Bounding the affine support function}\label{sec:sigma-bound}
In this section, we aim to establish uniform upper and strictly positive lower bounds for the affine support function \(\sigma\), which depend only on the initial curve \(\gamma_0\). We begin with the following lemma. 
\begin{lem}\label{inf-sup}
	If $\gamma$ is a smooth strictly convex curve enclosing area $A_0$, then there exists a point $x_0 \in \gamma$ such that
	\[
	\sigma(x_0) = \pi^{-\frac{2}{3}} A_0^{\frac{2}{3}}.
	\]
\end{lem}

\begin{proof}
	By Proposition~\ref{Iso ine} and \eqref{length-area}, we have
	\begin{align*}
		\sup \sigma 
		\ge \frac{\oint \sigma \, ds}{\mathcal{L}}
		= \frac{2A_0}{\mathcal{L}}
		\ge \pi^{-\frac{2}{3}} A_0^{\frac{2}{3}}.
	\end{align*}
	Suppose, for contradiction, that \(\sigma(x) > \pi^{-\frac{2}{3}} A_0^{\frac{2}{3}}\) for all \(x \in \gamma\).  
	Let $\mu = [\mathcal{T}, e]$ for a fixed nonzero vector $e \in \mathbb{R}^2$. Then $\mu$ vanishes precisely at the two points on $\gamma$ where $\mathcal{T}$ is parallel to $e$. A straightforward computation shows that
	\begin{equation}\label{mu}
		\mu_{ss} + \mathcal{K}\mu = 0.
	\end{equation}
	Let $f = \frac{\mu}{\sigma}$. Using \eqref{M-relation} and \eqref{mu}, we compute
	\begin{align*}
		f_{ss}
		&= \frac{\mu_{ss}}{\sigma}
		- \frac{f\,\sigma_{ss}}{\sigma}
		- \frac{2\sigma_s}{\sigma} f_s \\
		&= -\frac{f}{\sigma}
		- \frac{2\sigma_s}{\sigma} f_s.
	\end{align*}
	Introduce the parameter $d\ell = \sigma\, ds$. Then $\oint d\ell = 2A_0$, and
	\begin{align*}
		f_{ss} = \sigma^2 f_{\ell\ell} + \sigma \sigma_\ell f_\ell,
		\qquad 
		\sigma_s = \sigma \sigma_\ell,
		\qquad 
		f_s = \sigma f_\ell.
	\end{align*}
	Combining these identities yields
	\begin{equation*}
		\frac{\partial}{\partial \ell}
		\left( \sigma^3 \frac{\partial f}{\partial \ell} \right)
		+ f = 0.
	\end{equation*}
	If $\sigma > \pi^{-\frac{2}{3}} A_0^{\frac{2}{3}}$ everywhere, then by the Sturm comparison theorem (see \cite{Har64}), the distance in the $\ell$-parameter between two consecutive zeros of $f$ is strictly larger than $A_0$. Since $f$ has exactly two zeros, this implies that
	\[
	\oint d\ell > 2A_0,
	\]
	which contradicts $\oint d\ell = 2A_0$. The lemma follows.
\end{proof}

We now proceed to derive a priori estimates. 
Throughout this section, we assume that 
\(\gamma \in C^\infty(S^1 \times [0,T))\) is a strictly convex, origin-symmetric solution to the flow \eqref{main-flow1}. 
By Proposition~\ref{monotone}, the enclosed area remains equal to \(A_0\). 
In particular, Lemma~\ref{inf-sup} provides a normalization point for \(\sigma\), which will be used to control its oscillation.

\begin{prop}\label{bound-sigma-intergal}
	Let $\{\gamma(\cdot,t)\}_{t\in[0,T)}$ be a solution to \eqref{main-flow1}. Then 
	\[
	\frac{1}{
		\frac{\sqrt{2}}{2}\,\pi A_0^{-\frac{1}{2}}
		\left(\oint \frac{\sigma_s^2}{\sigma^2}\,ds\right)^{\frac{1}{2}}
		+ \pi^{\frac{2}{3}} A_0^{-\frac{2}{3}}
	}
	\le
	\sigma
	\le
	\left(
	\frac{\sqrt{2}}{4} A_0^{\frac{1}{2}}
	\left(\oint \frac{\sigma_s^2}{\sigma^2}\,ds\right)^{\frac{1}{2}}
	+ \pi^{-\frac{1}{3}} A_0^{\frac{1}{3}}
	\right)^2,
	\]
	on $S^1\times[0,T)$.
\end{prop}

\begin{proof}
	For $p\in\mathbb{R}$ to be specified later, define a weighted length element $d\xi = \sigma^p\, ds$. A straightforward computation yields
	\begin{equation*}
		\oint \left(\sigma^{\frac{p}{2}}\right)_\xi^2 \, d\xi
		= \frac{p^2}{4} \oint \frac{\sigma_s^2}{\sigma^2}\, ds.
	\end{equation*}
	Since
	\begin{align*}
		\text{osc}(\sigma^{\frac{p}{2}})\le\frac{1}{2}\left(\oint d\xi\right)^{\frac{1}{2}}\left(\oint \left(\sigma^{\frac{p}{2}}\right)_\xi^2 \, d\xi\right)^{\frac{1}{2}}.
	\end{align*}
	It follows that
	\begin{align*}
		\sup \left(\sigma^{\frac{p}{2}}\right) - \inf \left(\sigma^{\frac{p}{2}}\right)
		\le\frac{|p|}{4}\left(\oint d\xi\right)^{\frac{1}{2}}
		\left(\oint \frac{\sigma_s^2}{\sigma^2}\, ds\right)^\frac{1}{2}.
	\end{align*}
	We first choose $p=1$ in order to obtain an upper bound for \(\sigma\). In this case, $\oint d\xi = \oint \sigma\, ds = 2A_0$, and Lemma~\ref{inf-sup} implies that
	\[
	\sup \left(\sigma^{\frac{1}{2}}\right) - \pi^{-\frac{1}{3}} A_0^{\frac{1}{3}}
	\le \frac{\sqrt{2}}{4} A_0^{\frac{1}{2}}
	\left(\oint \frac{\sigma_s^2}{\sigma^2}\, ds\right)^\frac{1}{2}.
	\]
	This yields the upper bound
	\begin{align}\label{upper}
		\sigma \le \left(
		\frac{\sqrt{2}}{4}A_0^{\frac{1}{2}}
		\left(\oint \frac{\sigma_s^2}{\sigma^2}\, ds\right)^\frac{1}{2}
		+ \pi^{-\frac{1}{3}} A_0^{\frac{1}{3}}
		\right)^2.
	\end{align}
	Next, we choose $p=-2$ to derive a lower bound for $\sigma$. 
	By Proposition~\ref{BS ine}, we have
	\[
	\oint d\xi = \oint\frac{1}{\sigma^2}\, ds \le 2\pi^2 A_0^{-1}.
	\]
	Applying Lemma~\ref{inf-sup} again yields
	\[
	\sup (\sigma^{-1}) - \pi^{\frac{2}{3}} A_0^{-\frac{2}{3}}
	\le \frac{\sqrt{2}}{2}\,\pi A_0^{-\frac{1}{2}}
	\left(\oint \frac{\sigma_s^2}{\sigma^2}\, ds\right)^\frac{1}{2}.
	\]
	Hence, we get the desired lower bound
	\begin{align}\label{lower}
		\sigma \ge 
		\frac{1}{
			\frac{\sqrt{2}}{2}\,\pi A_0^{-\frac{1}{2}}
			\left(\oint \frac{\sigma_s^2}{\sigma^2}\, ds\right)^\frac{1}{2}
			+ \pi^{\frac{2}{3}} A_0^{-\frac{2}{3}}
		}.
	\end{align}
	Combining \eqref{upper} and \eqref{lower} completes the proof.
\end{proof}

To relate the integral quantity \(\oint \frac{\sigma_s^2}{\sigma^2} ds\) to a more geometric expression, we introduce
\begin{align*}
	\mathscr{L}=1-\frac{1}{2\pi^{\frac{4}{3}}}A^{\frac{1}{3}}\oint\mathcal{K}ds.
\end{align*}
By Proposition~\ref{AF ine} and Proposition~\ref{Iso ine}, the quantity \(\mathscr{L}\) is nonnegative.
Observe that
\begin{align}\label{totalcur}
	\oint \frac{\sigma_s^2}{\sigma^2}\, ds
	= \oint \frac{1}{\sigma}\, ds - \oint \mathcal{K}\, ds.
\end{align}
The first term on the right-hand side can be bounded in terms of \(A\). Indeed, by Hölder's inequality together with Propositions~\ref{Iso ine} and~\ref{BS ine}, we have
\begin{align}\label{1/sigma}
	\oint \frac{1}{\sigma}\, ds
	\le \mathcal{L}^{\frac{1}{2}}
	\left(\oint \frac{1}{\sigma^2}\, ds\right)^{\frac{1}{2}}
	\le 2\pi^{\frac{4}{3}} A^{-\frac{1}{3}}.
\end{align}
Combining \eqref{totalcur} and \eqref{1/sigma}, we obtain
\begin{align}\label{relate-l}
	\oint\frac{\sigma_s^2}{\sigma^2}\,ds
	\le 2\pi^{\frac{4}{3}} A^{-\frac{1}{3}}-\oint\mathcal{K}\,ds
	=2\pi^{\frac{4}{3}}A^{-\frac{1}{3}}\mathscr{L}.
\end{align}
As a consequence, this yields the following bounds for the affine support function.
\begin{cor}\label{Lsigma-bound}
	Let $\{\gamma(\cdot,t)\}_{t\in[0,T)}$ be a solution to \eqref{main-flow1}. Then
	\[
	\frac{1}{\pi^{\frac{5}{3}}A_0^{-\frac{2}{3}}\mathscr{L}^{\frac{1}{2}}
		+ \pi^{\frac{2}{3}} A_0^{-\frac{2}{3}}
	}
	\le
	\sigma
	\le
	\left(
	\frac{1}{2}A_0^{\frac{1}{3}}\pi^{\frac{2}{3}}\mathscr{L}^{\frac{1}{2}}	
	+ \pi^{-\frac{1}{3}} A_0^{\frac{1}{3}}
	\right)^2
	\]
	on $S^1\times[0,T)$.
\end{cor}

This corollary shows that uniform bounds for the affine support function \(\sigma\) reduce to controlling the quantity \(\mathscr{L}(t)\). We now establish such a bound.

\begin{prop}\label{l-bound}
	Let $\{\gamma(\cdot,t)\}_{t\in[0,T)}$ be a solution to \eqref{main-flow1}. Then there exists a constant $B = B(\gamma_0) > 0$ such that
	\[
	\mathscr{L}(t) < \mathscr{L}(0) + \frac{12}{11}B
	\]
	for all $t \in [0,T)$.
\end{prop}

\begin{proof}
	Let
	\[
	\mathscr{M}=1-\frac{A}{2\pi^2}\oint\frac{1}{\sigma^2}ds.
	\]
	By Propositions~\ref{BS ine} and \ref{monotone}, it follows that
	\begin{align*}
	0\le \mathscr{M}(t)\le \mathscr{M}(0)\le 1.
	\end{align*}
	We introduce the auxiliary quantity
	\[
	\mathscr{Q}(t)=\mathscr{L}(t)+B\mathscr{M}(t),
	\]
	where $B>0$ is a constant to be determined later. Then
	\[
	\mathscr{L}(t)\le \mathscr{Q}(t)\le \mathscr{L}(t)+B.
	\]
	For convenience, set $u=-\log\sigma$. Then
	\[
	\sigma_s=-e^{-u}u_s,\quad
	\sigma_{ss}=-e^{-u}u_{ss}+e^{-u}u_s^2,\quad
	\mathcal{K}=e^u+u_{ss}-u_s^2.
	\]
	By Proposition~\ref{evo}, we compute
	\begin{align*}
		\left(\oint \mathcal{K}\,ds\right)_t
		&=\frac{2}{3}\oint e^{-u}u_{sss}^2ds
		+\frac{10}{3}\oint e^{-u}u_s^2u_{ss}^2ds
		-2\oint e^{-u}u_su_{ss}u_{sss}ds\\
		&\quad+\frac{2}{15}\oint e^{-u}u_s^6ds
		-\frac{2}{3}\oint u_s^4ds
		-2\oint u_{ss}^2ds
		+\frac{4}{3}\oint e^uu_s^2ds,
	\end{align*}
	and
	\[
	\left(\oint\frac{1}{\sigma^2}\,ds\right)_t
	=6\oint e^{2u}u_s^2ds.
	\]
	Combining the above identities, we obtain
	\begin{align*}
		\mathscr{Q}_t
		&=\frac{A_0^{\frac{1}{3}}}{6\pi^{\frac{4}{3}}}
		\Bigl(
		-2\oint e^{-u}u_{sss}^2ds
		-10\oint e^{-u}u_s^2u_{ss}^2ds
		+6\oint e^{-u}u_su_{ss}u_{sss}ds\\
		&\quad-\frac{2}{5}\oint e^{-u}u_s^6ds
		+2\oint u_s^4ds
		+6\oint u_{ss}^2ds
		-4\oint e^uu_s^2ds\\
		&\quad-\frac{18A_0^{\frac{2}{3}}B}{\pi^{\frac{2}{3}}}\oint e^{2u}u_s^2ds
		\Bigr).
	\end{align*}
	Using the lower bound in Corollary~\ref{Lsigma-bound}
	\begin{align}\label{C(L)}
	e^u \ge \frac{1}{\left(
	\frac{1}{2}A_0^{\frac{1}{3}}\pi^{\frac{2}{3}}\mathscr{L}^{\frac{1}{2}}
	+ \pi^{-\frac{1}{3}} A_0^{\frac{1}{3}}
	\right)^2}
	=: \frac{1}{C(\mathscr{L})},
	\end{align}
	we deduce that
	\begin{align*}
		\mathscr{Q}_t
		&\le \frac{A_0^{\frac{1}{3}}}{6\pi^{\frac{4}{3}}}
		\Bigl(
		-2\oint e^{-u}u_{sss}^2ds
		-10\oint e^{-u}u_s^2u_{ss}^2ds
		+6\oint e^{-u}u_su_{ss}u_{sss}ds\\
		&\quad-\frac{2}{5}\oint e^{-u}u_s^6ds
		+2\oint u_s^4ds
		+6\oint u_{ss}^2ds
		-4\oint e^uu_s^2ds\\
		&\quad-\frac{18A_0^{\frac{2}{3}}B}{\pi^{\frac{2}{3}}C(\mathscr{L})}
		\oint e^{u}u_s^2ds
		\Bigr).
	\end{align*}
	Applying Cauchy--Schwarz and Young's inequalities, we estimate
	\begin{align*}
		6\oint e^{-u}u_su_{ss}u_{sss}ds
		&\le 6\epsilon_1\oint e^{-u}u_s^2u_{ss}^2ds
		+\frac{3}{2\epsilon_1}\oint e^{-u}u_{sss}^2ds,\\
		2\oint u_s^4ds
		&\le 2\epsilon_2\oint e^{-u}u_s^6ds
		+\frac{1}{2\epsilon_2}\oint e^uu_s^2ds,\\
		6\oint u_{ss}^2ds
		&\le 6\epsilon_3\oint e^{-u}u_{sss}^2ds
		+\frac{3}{2\epsilon_3}\oint e^uu_s^2ds.
	\end{align*}
	Choosing $\epsilon_1=\frac{5}{3}$, $\epsilon_2=\frac{1}{5}$, and $\epsilon_3=\frac{11}{60}$ yields
	\begin{align}\label{Q_t}
		\mathscr{Q}_t
		\le \frac{A_0^{\frac{1}{3}}}{6\pi^{\frac{4}{3}}}
		\left(
		\frac{5}{2}+\frac{90}{11}-4
		-\frac{18A_0^{\frac{2}{3}}B}{\pi^{\frac{2}{3}}C(\mathscr{L})}
		\right)
		\oint e^u u_s^2ds.
	\end{align}
	
	We now prove that
	\[
	\mathscr{L}(t) < \mathscr{L}(0) + \frac{12}{11}B
	\quad \text{for all } t \in [0,T).
	\]
	Assume, for contradiction, that there exists a first time $\bar t \in (0,T)$ such that
	\[
	\mathscr{L}(\bar{t}) = \mathscr{L}(0) + \frac{12}{11}B.
	\]
	Then $\mathscr{L}(t)\le \mathscr{L}(0)+\frac{12}{11}B$ for all $t\in[0,\bar t]$. From \eqref{C(L)} and \eqref{Q_t} we have
	\begin{align*}
		\mathscr{Q}_t
		\le \frac{A_0^{\frac{1}{3}}}{6\pi^{\frac{4}{3}}}
		\left(
		\frac{147}{22}
		-\frac{72B}{\pi^2\left(\mathscr{L}(0)+\frac{12}{11}B\right)
			+4+4\pi\left(\mathscr{L}(0)+\frac{12}{11}B\right)^{\frac{1}{2}}}
		\right)
		\oint e^u u_s^2ds
	\end{align*}
	for all $t\in[0,\bar t]$.
	Observing that
	\[
	\lim_{B\to\infty}
	\frac{72B}{\pi^2\left(\mathscr{L}(0)+\frac{12}{11}B\right)+4+4\pi\left(\mathscr{L}(0)+\frac{12}{11}B\right)^{\frac{1}{2}}}
	=\frac{66}{\pi^2},
	\]
	and that
	\[
	\frac{147}{22}-\frac{66}{\pi^2}<0.
	\]
	Hence, there exists $B>0$ such that
	\[
	\frac{72B}{\pi^2\left(\mathscr{L}(0)+\frac{12}{11}B\right)
		+4+4\pi\left(\mathscr{L}(0)+\frac{12}{11}B\right)^{\frac{1}{2}}}
	=\frac{147}{22}.
	\]
	In fact, $B$ can be solved explicitly as
	\begin{align*}
		B=\frac{11}{12}\left(\left(\frac{2\pi+\sqrt{4\pi^2+\left(4+\frac{484}{49}\mathscr{L}(0)\right)\left(\frac{484}{49}-\pi^2\right)}}{\frac{484}{49}-\pi^2}\right)^2-\mathscr{L}(0)\right).
	\end{align*}
	With this choice of $B$, we obtain $\mathscr{Q}_t\le 0$ on $[0,\bar t]$. It follows that
	\[
	\mathscr{Q}(\bar t)\le \mathscr{Q}(0)\le \mathscr{L}(0)+B,
	\]
	while
	\[
	\mathscr{Q}(\bar t)\ge \mathscr{L}(\bar t)
	=\mathscr{L}(0)+\frac{12}{11}B,
	\]
	which yields a contradiction. This completes the proof.
\end{proof}

Combining the previous estimates, we arrive at the desired uniform bounds for the affine support function.
\begin{cor}\label{bound-sigma}
	Let $\{\gamma(\cdot,t)\}_{t\in[0,T)}$ be a solution to \eqref{main-flow1}. Then
	\[
	\frac{1}{C}\le \sigma\le C,
	\]
	on $S^1 \times [0,T)$, where $C>0$ is a constant depending only on the initial curve $\gamma_0$.
\end{cor}

\section{Energy Estimates}\label{sec:energy-estimates}
In this section, we derive energy inequalities for smooth, strictly convex, origin-symmetric solutions of the flow \eqref{main-flow1}. 
Define the energy functional
\[
E(t) := \oint \sigma_s^2 \, ds.
\]
For computational convenience, we consider the following equivalent formulation of the flow \eqref{main-flow1}
\begin{align}\label{flow+tan}
	\frac{\partial \gamma}{\partial t} 
	= \sigma_{ss} \mathcal{N} - \sigma_{sss} \mathcal{T}, 
	\quad \gamma(\cdot,0) = \gamma_0.
\end{align}
That is, we take $\alpha = \sigma_{ss}$ and $\beta = -\alpha_s$ in \eqref{general-flow}. 
Since tangential components of the velocity affect only the parametrization, the flow \eqref{flow+tan} is equivalent to \eqref{main-flow1}.
By Proposition~\ref{evo}, the evolution equations under the flow \eqref{flow+tan} are given by
\begin{align*}\label{simga_t}
	\notag\frac{g_t}{g} 
	&= - \frac{2}{3}\sigma_{ssss}
	+ \frac{2}{3}\sigma^{-1}\sigma_{ss}^2  
	- \frac{2}{3}\sigma^{-1}\sigma_{ss}, \\
	\sigma_t 
	&= - \frac{1}{3}\sigma \sigma_{ssss} 
	+ \frac{1}{3}\sigma_{ss}^2 
	- \frac{4}{3}\sigma_{ss}.
\end{align*}
A direct computation yields
\begin{equation}\label{E-evo}
	\begin{aligned}
		\frac{dE}{dt}
		=&\oint\left(2\sigma_s\sigma_{st}+\sigma_s^2\frac{g_t}{g}\right)ds
		=\oint\left(2\sigma_s\sigma_{ts}-\sigma^2_s\frac{g_t}{g}\right)ds\\
		=&-\frac{2}{3}\oint\left(\sigma\sigma_s\sigma_{sssss}+\sigma_s^2\sigma_{ssss}-2\sigma_s\sigma_{ss}\sigma_{sss}+4\sigma_s\sigma_{sss}\right)ds\\
		&\quad+\frac{2}{3}\oint\left(\sigma_s^2\sigma_{ssss}-\sigma^{-1}\sigma_s^2\sigma_{ss}^2+\sigma^{-1}\sigma_s^2\sigma_{ss}\right)ds\\
		=& -\frac{2}{3}\oint \sigma \sigma_{sss}^2 \, ds
		+ \frac{1}{3}\oint \sigma_{ss}^3 \, ds
		- \frac{2}{3}\oint \sigma^{-1}\sigma_s^2 \sigma_{ss}^2 \, ds \\
		& \quad + \frac{2}{9}\oint \sigma^{-2}\sigma_s^4 \, ds
		+ \frac{8}{3}\oint \sigma_{ss}^2 \, ds.
	\end{aligned}
\end{equation}

We shall frequently use the interpolation inequality, which states that for a periodic function $ v $ with zero mean, the following inequality
\begin{equation*}\label{cz}
	\big\|v^{(j)}\big\|_{L^r}\le c\big\|v\big\|_{L^p}^{1-\lambda}\cdot\big\|v^{(k)}\big\|_{L^q}^\lambda, \quad \lambda\in(0,1)
\end{equation*}
is valid, where $ j,k,p,q $ and $ r$  satisfy $ p,q,r>1,\; j\ge0,  $
\begin{equation*}
	\frac{1}{r}=j+\lambda\left(\frac{1}{q}-k\right)+(1-\lambda)\frac{1}{p}, \qquad \frac{j}{k}\le \lambda\le 1,
\end{equation*}
and the constant $ c $ depends on $ j,k,p,q $  and $ r $ only.

We proceed to estimate the right-hand side of \eqref{E-evo}. By Corollary  \ref{bound-sigma}, we obtain
\begin{align}\label{0-order-0}
	\frac{dE}{dt} \le& -\frac{2}{3C}\oint \sigma_{sss}^2 \, ds
	+ \frac{1}{3}\oint \sigma_{ss}^3 \, ds 
	+ \frac{2C^2}{9}\oint \sigma_s^4 \, ds
	+ \frac{8}{3}\oint \sigma_{ss}^2 \, ds.
\end{align}
Applying Hölder's inequality, the interpolation inequality, and Young's inequality, we estimate each term as follows
\begin{equation}\label{0-order-1}
	\begin{aligned}
		\oint \sigma_{ss}^3 \, ds
		&\le \oint |\sigma_{ss}|^3 \, ds 
		\le C_1 \left(\oint \sigma_s^2 \, ds \right)^{\frac{5}{8}}
		\left(\oint \sigma_{sss}^2 \, ds \right)^{\frac{7}{8}} \\
		&\le C_1 \epsilon_1 \oint \sigma_{sss}^2 \, ds
		+ C_1 C(\epsilon_1)\left(\oint \sigma_s^2 \, ds \right)^5,
	\end{aligned}
\end{equation}
where $C(\epsilon_1)$ denotes a positive constant depending only on $\epsilon_1$.
Similarly,
\begin{equation}\label{0-order-2}
	\begin{aligned}
		\oint \sigma_s^4 \, ds
		&\le C_2 \left(\oint \sigma_s^2 \, ds \right)^{\frac{7}{4}}
		\left(\oint \sigma_{sss}^2 \, ds \right)^{\frac{1}{4}} \\
		&\le C_2 \epsilon_2 \oint \sigma_{sss}^2 \, ds
		+ C_2 C(\epsilon_2)\left(\oint \sigma_s^2 \, ds \right)^{\frac{7}{3}},
	\end{aligned}
\end{equation}
and
\begin{equation}\label{0-order-3}
	\begin{aligned}
		\oint \sigma_{ss}^2 \, ds
		&\le C_3 \left(\oint \sigma_s^2 \, ds \right)^{\frac{1}{2}}
		\left(\oint \sigma_{sss}^2 \, ds \right)^{\frac{1}{2}} \\
		&\le C_3 \epsilon_3 \oint \sigma_{sss}^2 \, ds
		+ C_3 C(\epsilon_3)\left(\oint \sigma_s^2 \, ds \right).
	\end{aligned}
\end{equation}
Substituting \eqref{0-order-1}--\eqref{0-order-3} into \eqref{0-order-0}, and choosing $\epsilon_1$, $\epsilon_2$, and $\epsilon_3$ sufficiently small, we obtain
\begin{align*}\label{0-order}
	\frac{dE}{dt} \le C\big(E + E^5\big),
\end{align*}
where $C$ depends only on the initial data and the constants in the interpolation inequalities.

Further, a more general energy inequality can be derived by an analogous argument. 
To this end, we introduce the notation
\[
f(\sigma)*\sigma_{s^{i_1}} * \sigma_{s^{i_2}} * \cdots * \sigma_{s^{i_m}},
\qquad
i_1 \le i_2 \le \cdots \le i_m
\]
to denote a generic term of the form
\[
K\, f(\sigma)\,\sigma_{s^{i_1}} \sigma_{s^{i_2}} \cdots \sigma_{s^{i_m}},
\]
where $K$ denotes a constant whose precise value is immaterial, and 
$f(\sigma)$ is a function of the form $f(\sigma)=\sigma^{-l}$ for some integer $l\ge0$. We want to establish the following energy inequality.

\begin{prop}\label{energy est.}
	For every integer $n \ge 1$, there exists a constant $C>0$ such that
	\begin{align}\label{m-order}
		\frac{d}{dt} \oint\sigma_{s^{n}}^2 \, ds
		\le C \bigl( E + E^{2n+3} \bigr),
	\end{align}
	where $C$ depends only on $n$, the initial data, and the constants in the interpolation inequalities.
\end{prop}

The proof of Proposition \ref{energy est.} relies on the following three elementary lemmas.

\begin{lem}\label{lem1}
	For any integer $n \ge 1$, we have
	\begin{align*}
		\sigma_{s^{n}t}
		&=
		-\frac{1}{3}\sigma\sigma_{s^{n+4}}
		+\sigma_s*\sigma_{s^{n+3}}
		\\
		&\qquad
		+
		\sum_{\substack{
				i_1+\cdots+i_{m}=n+4\\
				1\le i_1\le\cdots \le i_{m}\le n+2\\
				2\le m\le n+2
		}}
		f(\sigma)\sigma_{s^{i_1}} * \cdots * \sigma_{s^{i_{m}}}
		\\
		&\qquad\quad
		+
		\sum_{\substack{
				i_1+\cdots+i_{m}=n+2\\
				1\le i_1\le\cdots \le i_{m}\le n+1\\
				2\le m\le n+1
		}}
		f(\sigma)\sigma_{s^{i_1}} * \cdots * \sigma_{s^{i_{m}}}
		-\frac{4}{3}\sigma_{s^{n+2}}.
	\end{align*}
\end{lem}
\begin{proof}
	We argue by induction on $n$.
	For $n=1$, a direct computation yields
	\begin{align*}
		\sigma_{st}
		&=\sigma_{ts}-\frac{g_t}{g}\sigma_s\\
		&=-\frac{1}{3}\sigma\sigma_{sssss}
		+\frac{1}{3}\sigma_s\sigma_{ssss}
		+\frac{2}{3}\sigma_{ss}\sigma_{sss}
		-\frac{4}{3}\sigma_{sss}\\
		&\qquad
		-\frac{2}{3}\sigma^{-1}\sigma_s\sigma_{ss}^2
		+\frac{2}{3}\sigma^{-1}\sigma_s\sigma_{ss},
	\end{align*}
	which has the desired form.
	Assume that the statement holds for some $n\ge 1$. Then
	\begin{align*}
		\sigma_{s^{n+1}t}
		&=\sigma_{s^n t s}-\frac{g_t}{g}\sigma_{s^{n+1}}\\
		&=-\frac{1}{3}\sigma\sigma_{s^{n+5}}
		-\frac{1}{3}\sigma_s\sigma_{s^{n+4}}
		+\sigma_s*\sigma_{s^{n+4}}
		+\sigma_{ss}*\sigma_{s^{n+3}}\\
		&\qquad
		+
		\sum_{\substack{
				i_1+\cdots+i_{m}=n+4\\
				1\le i_1\le\cdots \le i_{m}\le n+2\\
				2\le m\le n+2
		}}
		f(\sigma)\sigma_s\sigma_{s^{i_1}} * \cdots * \sigma_{s^{i_{m}}}
		\\
		&\qquad\quad
		+
		\sum_{\substack{
				i_1+\cdots+i_{m}=n+5\\
				1\le i_1\le\cdots \le i_{m}\le n+3\\
				2\le m\le n+2
		}}
		f(\sigma)\sigma_{s^{i_1}} * \cdots * \sigma_{s^{i_{m}}}\\
		&\qquad\qquad
		+\frac{2}{3}\sigma_{ssss}\sigma_{s^{n+1}}
		-\frac{2}{3}\sigma^{-1}\sigma_{ss}^2\sigma_{s^{n+1}}
		\\
		&\quad\qquad\qquad
		+
		\sum_{\substack{
				i_1+\cdots+i_{m}=n+2\\
				1\le i_1\le\cdots \le i_{m}\le n+1\\
				2\le m\le n+1
		}}
		f(\sigma)\sigma_s\sigma_{s^{i_1}} * \cdots * \sigma_{s^{i_{m}}}
		\\
		&\qquad\qquad\qquad
		+
		\sum_{\substack{
				i_1+\cdots+i_{m}=n+3\\
				1\le i_1\le\cdots \le i_{m}\le n+2\\
				2\le m\le n+1
		}}
		f(\sigma)\sigma_{s^{i_1}} * \cdots * \sigma_{s^{i_{m}}}\\
		&\qquad\qquad\qquad\quad
		+\frac{2}{3}\sigma^{-1}\sigma_{ss}\sigma_{s^{n+1}}
		-\frac{4}{3}\sigma_{s^{n+3}}\\
		&=
		-\frac{1}{3}\sigma\sigma_{s^{n+5}}
		+\sigma_s*\sigma_{s^{n+4}}
		\\
		&\qquad
		+
		\sum_{\substack{
				i_1+\cdots+i_{m}=n+5\\
				1\le i_1\le\cdots \le i_{m}\le n+3\\
				2\le m\le n+3
		}}
		f(\sigma)\sigma_{s^{i_1}} * \cdots * \sigma_{s^{i_{m}}}
		\\
		&\qquad\quad
		+
		\sum_{\substack{
				i_1+\cdots+i_{m}=n+3\\
				1\le i_1\le\cdots \le i_{m}\le n+2\\
				2\le m\le n+2
		}}
		f(\sigma)\sigma_{s^{i_1}} * \cdots * \sigma_{s^{i_{m}}}
		-\frac{4}{3}\sigma_{s^{n+3}},
	\end{align*}
	which coincides with the desired formula with $n$ replaced by $n+1$. 
	This completes the proof.
\end{proof}

\begin{lem}\label{lem2}
	For every integer $n \ge 1$, the following identity holds:
	\begin{align*}
		\frac{d}{dt}\oint\sigma_{s^n}^2 \, ds
		=&-\frac{2}{3}\oint\sigma\sigma_{s^{n+2}}^2 \, ds+\oint\sigma_s*\sigma_{s^{n+1}}*\sigma_{s^{n+2}} \, ds
		\\
		&\quad
		+\sum_{\substack{
				i_1+\cdots+i_{m}=n+4\\
				1\le i_1\le\cdots \le i_{m}\le n+2\\
				2\le m\le n+2
		}}
		\oint f(\sigma)\sigma_{s^{i_1}} * \cdots * \sigma_{s^{i_{m}}}*\sigma_{s^n} \, ds
		\\
		&\quad
		+\sum_{\substack{
				i_1+\cdots+i_{m}=n+2\\
				1\le i_1\le\cdots \le i_{m}\le n+1\\
				2\le m\le n+1
		}}
		\oint f(\sigma)\sigma_{s^{i_1}} * \cdots * \sigma_{s^{i_{m}}}*\sigma_{s^n} \, ds
		\\
		&\quad
		+\frac{8}{3}\oint\sigma_{s^{n+1}}^2 \, ds.
	\end{align*}
\end{lem}

\begin{proof}
	Employing Lemma~\ref{lem1}, a straightforward computation yields
	\begin{align*}
		\frac{d}{dt}\oint\sigma_{s^n}^2\,ds
		&=2\oint\sigma_{s^n}\sigma_{s^n t}\,ds
		+\oint\sigma_{s^n}^2\frac{g_t}{g}\,ds \\
		&=-\frac{2}{3}\oint\sigma\sigma_{s^n}\sigma_{s^{n+4}}\,ds
		+\oint\sigma_s*\sigma_{s^n}*\sigma_{s^{n+3}}\,ds \\
		&\qquad
		+\sum_{\substack{
				i_1+\cdots+i_{m}=n+4\\
				1\le i_1\le\cdots \le i_{m}\le n+2\\
				2\le m\le n+2
		}}
		\oint f(\sigma)\sigma_{s^{i_1}} * \cdots * \sigma_{s^{i_{m}}}*\sigma_{s^n}\,ds
		\\
		&\qquad\quad
		+\sum_{\substack{
				i_1+\cdots+i_{m}=n+2\\
				1\le i_1\le\cdots \le i_{m}\le n+1\\
				2\le m\le n+1
		}}
		\oint f(\sigma)\sigma_{s^{i_1}} * \cdots * \sigma_{s^{i_{m}}}*\sigma_{s^n}\,ds
		\\
		&\qquad\qquad
		-\frac{8}{3}\oint\sigma_{s^n}\sigma_{s^{n+2}}\,ds
		-\frac{2}{3}\oint\sigma_{s^n}^2\sigma_{ssss}\,ds \\
		&\quad\qquad\qquad
		+\frac{2}{3}\oint\sigma_{s^n}^2\sigma^{-1}\sigma_{ss}^2\,ds
		-\frac{2}{3}\oint\sigma_{s^n}^2\sigma^{-1}\sigma_{ss}\,ds.
	\end{align*}
	By integration by parts, we obtain the following identities
	\begin{align*}
		\oint\sigma\sigma_{s^n}\sigma_{s^{n+4}}\,ds
		=\oint\sigma\sigma_{s^{n+2}}^2\,ds
		+\oint\sigma_{ss}\sigma_{s^n}\sigma_{s^{n+2}}\,ds 
		+2\oint\sigma_s\sigma_{s^{n+1}}\sigma_{s^{n+2}}\,ds,
	\end{align*}
	\begin{align*}
		\oint\sigma_s*\sigma_{s^n}*\sigma_{s^{n+3}}\,ds
		=\oint\sigma_{ss}*\sigma_{s^n}*\sigma_{s^{n+2}}\,ds
		+\oint\sigma_s*\sigma_{s^{n+1}}*\sigma_{s^{n+2}}\,ds,
	\end{align*}
	\begin{align*}
		\oint\sigma_{s^n}\sigma_{s^{n+2}}\,ds
		=-\oint\sigma_{s^{n+1}}^2\,ds.
	\end{align*}
	Substituting these into the previous expression and collecting like terms, we arrive at the desired formula.
	This completes the proof.
\end{proof}

\begin{lem}\label{HiY ine}
	For every integer $m\ge1$, $l\ge1$, and $k\ge0$, and for every $\epsilon>0$, we have
	\begin{align*}
		\oint|\sigma_{s^{i_1}}\cdots\sigma_{s^{i_m}}| ds
		\le
		\epsilon\oint\sigma_{s^{k+1}}^2 ds
		+
		C E^{\frac{2km-2l+m+2}{4k-2l+m+2}},
	\end{align*}
	where $i_1+\cdots+i_m=l$, 
	$1\le i_1\le\cdots\le i_m$, and $C$ depends on $\epsilon$, $m$, $l$, $k$.
\end{lem}
\begin{proof}
	By H\"{o}lder's inequality,
	\begin{align}\label{Holder}
		\oint\lvert\sigma_{s^{i_1}}\cdots\sigma_{s^{i_m}}\rvert ds\le\prod\limits_{j=1}^m\left(\oint\left|\sigma_{s^{i_j}}\right|^mds\right)^{\frac{1}{m}}.
	\end{align}
	Using the interpolation inequality, each term can be estimated as follows
	\begin{align*}
		\left(\oint\left|\sigma_{s^{i_j}}\right|^mds\right)^{\frac{1}{m}}\le c_j E^{\frac{1}{2}\left(1-\frac{i_j}{k}+\frac{1}{2k}+\frac{1}{mk}\right)}\left(\oint\sigma_{s^{k+1}}^2ds\right)^{\frac{1}{2}\left(\frac{i_j}{k}-\frac{1}{2k}-\frac{1}{mk}\right)},
	\end{align*}
	where $ c_j $ is a positive constant.
	Substituting into \eqref{Holder}, we obtain
	\begin{align*}
		\oint\left|\sigma_{s^{i_1}}\cdots\sigma_{s^{i_m}}\right| ds & \le cE^{\frac{1}{2}\left(m-\frac{l}{k}+\frac{m}{2k}+\frac{1}{k}\right)}\left(\oint\sigma_{s^{k+1}}^2ds\right)^{\frac{1}{2}\left(\frac{l}{k}-\frac{m}{2k}-\frac{1}{k}\right)}\\ & \le \epsilon\oint\sigma_{s^{k+1}}^2ds+C(\epsilon,m,l,k)E^\frac{2km-2l+m+2}{4k-2l+m+2},
	\end{align*}
	where $ \epsilon $ is an arbitrary positive constant.
\end{proof}

\begin{proof}[Proof of Proposition \ref{energy est.}]
	By Lemma~\ref{lem2} and Corollary~\ref{bound-sigma}, we obtain
	\begin{align*}
		\frac{d}{dt}\oint\sigma_{s^n}^2 ds
		\le&
		-C_0\oint\sigma_{s^{n+2}}^2 ds
		+C_1\oint|\sigma_s\sigma_{s^{n+1}}\sigma_{s^{n+2}}| ds
		+\frac{8}{3}\oint\sigma_{s^{n+1}}^2 ds\\
		&\qquad\quad
		+\sum_{\substack{
				i_1+\cdots+i_{m}=n+4\\
				1\le i_1\le\cdots \le i_{m}\le n+2\\
				2\le m\le n+2
		}}
		C_2\oint|\sigma_{s^{i_1}} \cdots  \sigma_{s^{i_{m}}}\sigma_{s^n}| ds
		\\
		&\qquad\qquad
		+\sum_{\substack{
				i_1+\cdots+i_{m}=n+2\\
				1\le i_1\le\cdots \le i_{m}\le n+1\\
				2\le m\le n+1
		}}
		C_3\oint|\sigma_{s^{i_1}}  \cdots  \sigma_{s^{i_{m}}}\sigma_{s^n}| ds,
	\end{align*}
	where $C_i\ge0$ are constants.
	We estimate each term using Lemma~\ref{HiY ine}. For any $\epsilon_i>0$, we have
	\begin{align*}
		C_1\oint|\sigma_s\sigma_{s^{n+1}}\sigma_{s^{n+2}}| ds
		\le
		\epsilon_1\oint\sigma_{s^{n+2}}^2 ds
		+
		c_1 E^{2n+3},
	\end{align*}
	\begin{align*}
		\sum C_2\oint|\sigma_{s^{i_1}} \cdots  \sigma_{s^{i_{m}}}\sigma_{s^n}| ds
		&\le
		\epsilon_2\oint\sigma_{s^{n+2}}^2 ds
		+
		c_2 E^{2n+3},\\
		\sum C_3\oint|\sigma_{s^{i_1}}  \cdots  \sigma_{s^{i_{m}}}\sigma_{s^n}| ds
		&\le
		\epsilon_3\oint\sigma_{s^{n+2}}^2 ds
		+
		c_3 E^{2n+3-\frac{8(n+1)}{m+3}},
	\end{align*}
	\begin{align*}
		\frac{8}{3}\oint\sigma_{s^{n+1}}^2 ds
		\le
		\epsilon_4\oint\sigma_{s^{n+2}}^2 ds
		+
		c_4 E.
	\end{align*}
Choosing $\epsilon_i>0$ sufficiently small so that
$-C_0 + \sum_{i=1}^4 \epsilon_i \le 0$, and noting that
$1 \le 2n+3 - \frac{8(n+1)}{m+3} \le 2n+3$ for $m \ge 1$, we conclude that
	\begin{align*}
		\frac{d}{dt} \oint \sigma_{s^{n}}^2 ds
		\le
		C \bigl( E + E^{2n+3} \bigr),
	\end{align*}
	where $C$ depends only on $n$, the initial data $\gamma_0$, and the constants in the interpolation inequalities.
	This completes the proof.
\end{proof}

\section{Proof of Theorem \ref{A}}\label{Proof of Theorem A}

We begin by deriving several consequences of the energy inequalities.

\begin{prop}\label{E-bound}
	Let $\{\gamma(\cdot,t)\}_{t\in[0,T)}$ be a smooth, strictly convex, origin-symmetric solution to the flow \eqref{main-flow1}.
	Then the energy $E(t)$ is uniformly bounded in terms of the initial curve $\gamma_0$.
\end{prop}

\begin{proof}
	By Corollary~\ref{bound-sigma}, we have
	\[
	E(t) = \oint \sigma_s^2 \, ds
	= \oint \frac{\sigma_s^2}{\sigma^2} \, \sigma^2 \, ds
	\le C \oint \frac{\sigma_s^2}{\sigma^2} \, ds,
	\]
	where $C>0$ depends only on $\gamma_0$.
	Using \eqref{relate-l} and Proposition~\ref{l-bound}, we then get
	\[
	E(t) \le C\, 2\pi^{\frac{4}{3}} A_0^{-\frac{1}{3}} \left(\mathscr{L}(0) + \frac{12}{11}B\right),
	\]
	for some $B>0$ depending only on $\gamma_0$. 
	This completes the proof.
\end{proof}

\begin{cor}\label{higherorder-sigma}
	Let $\{\gamma(\cdot,t)\}_{t\in[0,T)}$ be a smooth, strictly convex, origin-symmetric solution to the flow \eqref{main-flow1}.
	Then for every $m \in \mathbb{N}$, the quantity $\oint \sigma_{s^m}^2 \, ds$ is uniformly bounded in terms of the initial curve $\gamma_0$.
\end{cor}

\begin{proof}
	By \eqref{m-order} and Proposition~\ref{E-bound}, we have
	\begin{align*}
		\left(\oint \sigma_{s^m}^2 \, ds\right)(t)
		&\le \left(\oint \sigma_{s^m}^2 \, ds\right)(0)
		+ C \int_{0}^{t} E\big(E^{2m+2} + 1\big)\, dt \\
		&\le \left(\oint \sigma_{s^m}^2 \, ds\right)(0)
		+ C_1 \int_{0}^{t} E \, dt,
	\end{align*}
	where $C_1$ depends only on the initial curve $\gamma_0$.
	By Corollary~\ref{bound-sigma} and \eqref{BS-quantity}, we estimate
	\begin{align*}
		\int_{0}^{t} E \, dt
		&= \int_{0}^{t} \left( \oint \sigma^4 \frac{\sigma_s^2}{\sigma^4} \, ds \right) dt \\
		&\le C_2 \int_{0}^{t} \oint \frac{\sigma_s^2}{\sigma^4} \, ds \, dt \\
		&= \frac{C_2}{6} \left( \left(\oint \frac{1}{\sigma^2} \, ds\right)(t)
		- \left(\oint \frac{1}{\sigma^2} \, ds\right)(0) \right),
	\end{align*}
	where $C_2$ depends only on the initial curve $\gamma_0$.
	It follows from Proposition~\ref{BS ine} that the quantity $\int_{0}^{t} E \, dt$ is bounded. Consequently, $\oint \sigma_{s^m}^2 \, ds$ is uniformly bounded on $[0,T)$, which completes the proof.
\end{proof}
These estimates allow us to show that the solution preserves strict convexity as long as it exists.
\begin{prop}[Preservation of strict convexity]
	Let $\{\gamma(\cdot,t)\}_{t\in[0,T)}$ be a smooth solution to the flow \eqref{main-flow1} with smooth, origin-symmetric, closed, strictly convex initial curve $\gamma_0$.
	Then $\gamma(\cdot,t)$ remains strictly convex for all $t \in [0,T)$.
\end{prop}
\begin{proof}
	We first derive the evolution equation for the radius of curvature $r$. 
	By \eqref{main-flow2}, \eqref{vector-relation}, and \eqref{support-relation}, 
	the scalar equation \eqref{scalar-eq} can be rewritten as
	\begin{align}\label{ht}
		\frac{\partial h}{\partial t} 
		= -r^{-\frac{1}{3}}\sigma_{ss}
		= -h\sigma^{-1}\sigma_{ss}.
	\end{align}
	A direct computation yields
	\begin{align*}
		\frac{\partial(\log r)}{\partial t}
		&= r^{-1}(h_{t\theta\theta}+h_t)\\
		&= -r^{-1}\Big((r^{-\frac{1}{3}})_{\theta\theta}\sigma_{ss}
		+ r^{-\frac{1}{3}}\sigma_{ss}
		+ r^{-\frac{1}{3}}\sigma_{ss\theta\theta}
		+ 2(r^{-\frac{1}{3}})_\theta\sigma_{ss\theta}\Big).
	\end{align*}
	Using \eqref{elements}, we obtain
	\begin{align*}
		r^{-1}\sigma_{ss\theta\theta}
		= r\sigma_{ssss}+\frac{2}{3}r_s\sigma_{sss},
	\end{align*}
	and
	\begin{align*}
		2\left(r^{-\frac{1}{3}}\right)_\theta\sigma_{ss\theta}
		= -\frac{2}{3}r_s\sigma_{sss}.
	\end{align*}
	Substituting these identities into the above expression, 
	and using \eqref{curva-relation} together with \eqref{M-relation}, we arrive at
	\begin{equation}\label{logrt}
		\begin{aligned}
			\frac{\partial(\log r)}{\partial t}
			&= -\mathcal{K}\sigma_{ss}
			- \sigma_{ssss}
			- \frac{2}{3}r^{-1}r_s\sigma_{sss}
			+ \frac{2}{3}r^{-1}r_s\sigma_{sss}\\
			&= -\sigma_{ssss}
			+ \sigma^{-1}\sigma_{ss}^2
			- \sigma^{-1}\sigma_{ss}.
		\end{aligned}
	\end{equation}
	
	We now prove the claim. Suppose, for contradiction, that there exists a first time $\bar{t} \in (0,T)$ such that $\gamma(\cdot,\bar{t})$ is not strictly convex. Then $\gamma(\cdot,t)$ is strictly convex for all $t \in [0,\bar{t})$.
	By Corollary~\ref{bound-sigma}, $\sigma^{-1}$ is uniformly bounded on $S^1 \times [0,\bar{t})$. 
	Moreover, by Corollary~\ref{higherorder-sigma} and the Sobolev inequality, $\sigma$ and all its derivatives are uniformly bounded on $S^1 \times [0,\bar{t})$. 
	It follows from \eqref{logrt} that $\partial_t (\log r)$ is uniformly bounded. 
	Hence there exists a constant $C>0$ such that
	\[
	r_{\min}(0)e^{-Ct} \le r(\cdot,t) \le r_{\max}(0)e^{Ct}
	\quad \text{for all } t \in [0,\bar{t}).
	\]
	By smoothness, these bounds persist at $t=\bar{t}$, so that $\kappa(\cdot,\bar{t})>0$. 
	This contradicts the definition of $\bar{t}$.
	Thus the proof is complete.
\end{proof}

We now show that the solution exists for all time.
\begin{prop}\label{longtime}
	Let $\gamma_0$ be a smooth, origin-symmetric, closed, strictly convex curve. 
	Then the solution of the flow \eqref{main-flow1} exists on $[0,+\infty)$.
\end{prop}
\begin{proof}
	Suppose, for contradiction, that the maximal existence time $T$ is finite.  
	By Corollary~\ref{bound-sigma}, $\sigma^{-1}$ is uniformly bounded on $S^1 \times [0,T)$, and by Corollary~\ref{higherorder-sigma} together with the Sobolev inequality, $\sigma$ and all its derivatives are uniformly bounded.  
	It follows from \eqref{logrt} that $\partial_t (\log r)$ is uniformly bounded, so $r$ remains bounded above and below on $[0,T)$.  
	Similarly, the evolution equation \eqref{ht} implies that $\partial_t (\log h)$ is uniformly bounded, hence $h$ is also uniformly bounded. 
	Therefore, $\gamma(\cdot,t)$ is uniformly bounded in $C^k$ for every $k \ge 0$.  
	By the Arzelà--Ascoli theorem, there exists a sequence $t_j \to T$ such that $\gamma(\cdot,t_j)$ converges to a limit $\gamma(\cdot,T)$, which is a strictly convex smooth closed curve.  
	Since the time derivative of $\gamma$ and all its derivatives remain bounded, $\gamma(\cdot,t)$ converges in $C^\infty$ to $\gamma(\cdot,T)$ as $t \to T$.  
	The short-time existence result (Proposition~\ref{short-time}) then shows that the solution can be extended beyond $T$, contradicting the maximality of $T$.  
\end{proof}

The following result shows that the energy $E(t)$ decays to zero as $t \to +\infty$.
\begin{prop}\label{E-decay}
	Let $\gamma(\cdot,t)$ be a solution to the flow \eqref{main-flow1}. 
	Then
	\[
	\lim_{t \to +\infty} E(t) = 0.
	\]
\end{prop}

\begin{proof}
	By \eqref{BS-quantity} and Corollary~\ref{bound-sigma}, 
	\[
	\int_{0}^{+\infty} \left( \oint \frac{1}{\sigma^2} \, ds \right)_t dt 
	\ge C_1 \int_{0}^{+\infty} \oint \sigma_s^2 \, ds \, dt
	= C_1 \int_{0}^{+\infty} E(t) \, dt,
	\]
	where $C_1$ is a positive constant depending only on $\gamma_0$.
	On the other hand, Proposition~\ref{BS ine} implies that the left-hand side is bounded by $2\pi^2A_0^{-1}$.
	Consequently, we obtain
	\begin{align}\label{int-con}
		\int_{0}^{+\infty} E(t) \, dt \le C,
	\end{align}
	for some constant $C>0$ depending only on $\gamma_0$. Moreover, combining \eqref{E-evo} with Proposition~\ref{bound-sigma}, Corollary~\ref{higherorder-sigma}, and the Sobolev inequality, it follows that the time derivative $\frac{dE}{dt}$ is uniformly bounded.
	
	We argue by contradiction. Suppose that $E(t)$ does not converge to zero as $t \to +\infty$. Then there exist a constant $C_0>0$ and a sequence $\{t_i\}_{i=1}^\infty$ with $t_i \to +\infty$ such that $E(t_i) \ge C_0$ for all $i$. 
	Since $\frac{dE}{dt}$ is uniformly bounded, there exists $\epsilon_0>0$ such that $E(t) \ge \frac{C_0}{2}$ for all $t \in (t_i, t_i+\epsilon_0)$. 
	Integrating over these intervals yields
	\begin{align}\label{cont}
		\int_{t_i}^{t_i+\epsilon_0} E(t) \, dt \ge \frac{C_0 \epsilon_0}{2}.
	\end{align}
	However, \eqref{int-con} implies
	\[
	\lim_{i \to \infty} \int_{t_i}^{+\infty} E(t)\,dt = 0,
	\]
	which contradicts \eqref{cont}. Hence, $\lim_{t \to +\infty} E(t) = 0$, as claimed.
\end{proof}

With this decay in hand, we turn to the asymptotic behavior of the affine support function.
\begin{prop}\label{sigma-converge}
	Let $\gamma(\cdot,t)$ be a solution to the flow \eqref{main-flow1} with a smooth, origin-symmetric, closed, strictly convex initial curve $\gamma_0$. 
	Then the affine support function $\sigma(\cdot,t)$ converges smoothly to the constant 
	$\pi^{-\frac{2}{3}} A_0^{\frac{2}{3}}$
	as $t \to +\infty$.
\end{prop}

\begin{proof}
	Corollary~\ref{higherorder-sigma} together with the Sobolev inequality shows that 
	$\sigma$ and all its spatial derivatives are uniformly bounded.  
	It then follows from the Arzelà--Ascoli theorem that, for any sequence $\{t_j\} \to +\infty$, 
	there exists a subsequence $\{t_{j_i}\}$ along which $\sigma(\cdot,t_{j_i})$ converges smoothly to some function $f$.  
	Proposition~\ref{E-decay} tells us that $f$ must be constant.
	Because the flow \eqref{main-flow1} preserves the enclosed area, 
	Lemma~\ref{inf-sup} identifies this constant as 
	$\pi^{-\frac{2}{3}} A_0^{\frac{2}{3}}$.
	Putting this together, we see that every subsequence has the same limit, 
	so $\sigma(\cdot,t)$ converges smoothly to 
	$\pi^{-\frac{2}{3}} A_0^{\frac{2}{3}}$ as $t \to +\infty$.
\end{proof}

Finally, we deduce the convergence of the evolving curves up to $\mathrm{SL}(2)$.
\begin{prop}\label{converge}
	Let $\gamma_0: S^1 \to \mathbb{R}^2$ be a smooth, strictly convex curve symmetric with respect to the origin. 
	Then the solution $\gamma(\cdot,t)$ of \eqref{main-flow1} converges smoothly, as $t \to +\infty$, to a round circle 
	with the same enclosed area as $\gamma_0$, modulo $\mathrm{SL}(2)$.
\end{prop}

\begin{proof}
	Let $\{L_t\}_{t \ge 0} \subset \mathrm{SL}(2)$ be a family of special linear transformations chosen so that the Euclidean length of $L_t\gamma(\cdot,t)$ is minimized.  
	By Propositions~\ref{symmetric} and \ref{monotone}, the curve $L_t\gamma(\cdot,t)$ remains origin-symmetric and preserves the enclosed area throughout the flow \eqref{main-flow1}.  
	Proposition~8 of \cite{And99} then ensures that the support function $h$ of $L_t\gamma(\cdot,t)$ is bounded above and below by positive constants.  
	Since the affine support function is invariant under $\mathrm{SL}(2)$, \eqref{support-relation} together with Corollary~\ref{bound-sigma} implies that the radius of curvature $r$ is also bounded above and below.  
	Corollary~\ref{higherorder-sigma} and the Sobolev inequality further guarantee that all derivatives of $\sigma$ are uniformly bounded, so the curves $L_t\gamma(\cdot,t)$ are uniformly bounded in $C^k$ for every $k$.  
	It follows from the Arzelà–Ascoli theorem that for any sequence $t_i \to +\infty$, there exists a subsequence along which $L_{t_{i_k}}\gamma(\cdot,t_{i_k})$ converges smoothly to a limit curve $\gamma_\infty$.  
	By Proposition~\ref{sigma-converge}, this limit is an ellipse with area $A_0$, and the length-minimizing condition then forces $\gamma_\infty$ to be a circle with the same area.  
	Since every subsequence converges to this circle, we conclude that $L_t\gamma(\cdot,t)$ converges smoothly to it as $t \to +\infty$.
\end{proof}

\begin{proof}[Proof of Theorem~\ref{A}]
	The theorem follows by combining Propositions \ref{longtime} and \ref{converge}.
\end{proof}

\section*{Acknowledgements}
This work is supported by the National Natural Science Foundation of China (No. 12571062).

\end{document}